\documentclass[12pt,a4paper,final]{amsart}
\usepackage{amssymb}
\usepackage{hyperref}
\usepackage[notcite, notref]{showkeys}
\usepackage[utf8]{inputenc}

\theoremstyle{plain}
\newtheorem{theorem}{Theorem}[section]
\newtheorem{proposition}[theorem]{Proposition}
\newtheorem{corollary}[theorem]{Corollary}
\newtheorem{lemma}[theorem]{Lemma}
\newtheorem{claim}{Claim}

\theoremstyle{definition}

\newtheorem{remark}[theorem]{Remark}
\newtheorem{question}[theorem]{Question}

\newtheorem{algorithm}[theorem]{Algorithm}

\theoremstyle{remark}

\numberwithin{equation}{section}

\newcommand{\N}{\mathbb N}
\newcommand{\Z}{\mathbb Z}

\newcommand{\R}{\mathbb R}
\newcommand{\C}{\mathbb C}
\newcommand{\F}{\mathbb F}

\DeclareMathOperator{\GL}{GL}
\DeclareMathOperator{\SL}{SL}
\DeclareMathOperator{\Ot}{O}

\newcommand{\op}{\operatorname}
\newcommand{\Id}{\textup{Id}}

\newcommand{\mi}{\mathrm{i}}

\DeclareMathOperator{\Spec}{Spec}
\DeclareMathOperator{\Eig}{Eig}

\DeclareMathOperator{\Aut}{Aut}

\DeclareMathOperator{\vol}{vol}

\title
[The fundamental group]{The fundamental group of a spherical space form is not audible}

\author{Mauro Colantonio}
\address{Instituto de Matemática (INMABB), Departamento de Matemática, Universidad Nacional del Sur (UNS)-CONICET, Bahía Blanca, Argentina.}
\email{mauro.colantonio@uns.edu.ar}

\author{Emilio~A.~Lauret}
\address{Instituto de Matemática (INMABB), Departamento de Matemática, Universidad Nacional del Sur (UNS)-CONICET, Bahía Blanca, Argentina.}
\email{emilio.lauret@uns.edu.ar}

\thanks{The authors were supported by grants from SGCYT--UNS (PGI 24/L126) and FONCyT (PICT-2019-01054).}
\date{\today}

\begin{document}

\begin{abstract}
We revisit the problem of isospectral spherical space forms with non-cyclic fundamental groups after the works by Ikeda, Gilkey and Wolf. 
We find the first pair of spherical space forms with non-isomorphic fundamental groups and the same Laplace spectrum. 
This shows that the isomorphism class of the fundamental group is not audible among spherical space forms. 
We also found several instances where one can hear the fundamental group of a spherical space form (among spherical space forms). 
\end{abstract} 

\maketitle

	
\section{Introduction}

Every Riemannian manifold $(M,g)$ has an associated distinguished differential operator, the Laplace--Beltrami operator $\Delta_g$.
When $M$ is compact, the spectrum $\Spec(M,g)$ of $\Delta_g$ (that is, the collection of eigenvalues of $\Delta_g$ counted with multiplicities) is real, non-negative, and discrete.
In particular, each eigenvalue has finite multiplicity.

Two compact Riemannian manifolds $(M,g)$ and $(M',g')$ are said to be \emph{isospectral} if their spectra coincide, that is, if $\Spec(M,g) = \Spec(M',g')$.
Following the evocative title of Kac's article~\cite{Kac66}, a geometric or topological property of $(M,g)$ is said to be \emph{audible} (also called a \emph{spectral invariant}) if it is determined by $\Spec(M,g)$.
For instance, the dimension, the volume, and the total scalar curvature are audible features.
This means that two isospectral manifolds have the same dimension, volume, and total scalar curvature.

In this article we are interested in audible properties of \emph{spherical space forms}, that is, manifolds of the form $S^d/G$, where $G$ is a subgroup of $\Ot(d+1)$ acting freely on the unit sphere $S^d$, endowed with constant positive sectional curvature.
The group $G$ is always finite and isomorphic to the fundamental group of $S^d/G$.

Given two isospectral spherical space forms $S^{d_1}/G_1$ and $S^{d_2}/G_2$, the spectral invariants mentioned above imply that
$
d := d_1 = d_2,
$
and that $G_1$ and $G_2$ have the same cardinality, since
$$
\frac{\vol(S^d)}{|G_1|}
= \vol(S^d/G_1)
= \vol(S^d/G_2)
= \frac{\vol(S^d)}{|G_2|}.
$$
Moreover, both manifolds have the same constant positive sectional curvature, which we assume to be equal to one from now on without loss of generality.

\begin{remark}
We emphasize that we are looking for audible properties of spherical space forms within the class of spherical space forms. 
For instance, the article \cite{LinSchmidtSuttonII} by Lin, Schmidt and Sutton studies isospectrality among locally homogeneous metrics on arbitrary elliptic $3$-manifolds. 
\end{remark}

In Wolf's classification of spherical space forms (see~\cite[Ch.~7]{Wolf-book}), the possible fundamental groups fall into six disjoint types, namely Types~I--VI.
Those of Type~I (i.e.\ all their Sylow subgroups are cyclic) are the most abundant.
For instance, the fundamental group of a spherical space form of dimension congruent to $1 \pmod 4$ is necessarily of Type~I.

Ikeda~\cite{Ikeda80_isosp-lens} constructed the first examples of non-isometric isospectral spherical space forms.
They consisted of families of lens spaces (i.e.\ their fundamental groups are cyclic, and hence of Type~I).
Many new examples of isospectral lens spaces have appeared since then, exhibiting different features.
We refer the reader to the survey article~\cite{LMR-SaoPaulo} for further details.

Ikeda~\cite{Ikeda83} also constructed non-isometric isospectral spherical space forms with non-cyclic fundamental groups. 
Roughly speaking, he proved that if $\Gamma$ is a non-cyclic Type I group, which is fixed point free (i.e.\ it has a real representation $\rho:\Gamma\to\Ot(d+1)\subset\GL(d+1)$ such that $\rho(\Gamma)$ acts freely on $S^{d}$), then all spherical space forms with fundamental group isomorphic to $\Gamma$ and minimal dimension are pairwise isospectral. 

\begin{theorem}[\cite{Ikeda83}] \label{thm1:Ikeda}
Let $\Gamma$ be a non-cyclic fixed point free Type I group. 
Given two irreducible real representations $\rho_1,\rho_2$ of $\Gamma$ such that $\rho_1(\Gamma)$ and $\rho_2(\Gamma)$ act freely on the unit sphere $S^d$ (they have necessarily the same dimension $d+1$), then $\Spec(S^d/\rho_1(\Gamma)) =\Spec(S^d/\rho_2(\Gamma))$.
\end{theorem}

As a consequence, Ikeda obtained in \cite[Thm.~4]{Ikeda83} that there are isospectral and non-isometric spherical space forms in every odd dimension $\geq5$.
(The proof of the necessary result \cite[Lem.~2.5]{Ikeda83} is corrected in \cite[Lem.~1]{Ikeda97}.)

A few years later, Gilkey~\cite{Gilkey85} proved that the same examples in \cite{Ikeda83} are \emph{strongly isospectral}, that is, any natural operator acting on the same natural bundle on both manifolds have the same spectra. 
Strongly isospectral manifolds are isospectral with respect to the Hodge-Laplace operators on $p$-forms (called \emph{$p$-isospectral}) for all $p$. 
The converse is not true due to the existence of pairs of non-isometric lens spaces $p$-isospectral for all $p$ but not strongly isospectral constructed in \cite{LMR-onenorm}. 

\begin{remark}\label{rem:stronglyisospectral}
All the examples of isospectral spherical space forms with non-cyclic fundamental groups known so far, including those in this article, are strongly isospectral. 
\end{remark}

Later, Wolf~\cite{Wolf01} extended the construction in  Theorem~\ref{thm1:Ikeda} to any fixed point free group $\Gamma$. 
Roughly speaking, he defined a group of transformations $\mathcal A(d,\Gamma)$ on the finite set $\mathcal S(d,\Gamma)$ of all isometry classes of $d$-dimensional spherical space forms with fundamental group isomorphic to $\Gamma$ such that two elements in the same $\mathcal A(d,\Gamma)$-orbit are isospectral. 
Moreover, if $d$ is minimal for $\Gamma$, the action is transitive, so any two spaces in $\mathcal S(d,\Gamma)$ are isospectral. 
The precise definition of $\mathcal A(d,\Gamma)$ is very technical, and it depends deeply on the classification of spherical space forms (see \cite[Def.~6.1]{Wolf01}). 

The results mentioned above from~\cite{Ikeda83}, \cite{Gilkey85}, and~\cite{Wolf01} are, so far, the only known examples of isospectral spherical space forms with non-cyclic fundamental groups.
Because of their construction, any such isospectral pair has isomorphic fundamental groups.
Moreover, two isospectral lens spaces obviously have isomorphic fundamental groups, and a lens space cannot be isospectral to a spherical space form with non-cyclic fundamental group.
These facts lead us to ask whether the fundamental group is audible within the class of spherical space forms.

\begin{question}\label{question}
	Are there isospectral spherical space forms with non-isomorphic fundamental groups? 
\end{question}

As a partial negative answer, Ikeda~\cite[Thm.~3.1]{Ikeda80_3-dimII} proved that, for any odd prime $d$, two isospectral spherical space forms of dimension $2d-1$ must have isomorphic fundamental groups.
This result, combined with Theorem~\ref{thm1:Ikeda}, explains all possible isospectralities among spherical space forms with non-cyclic fundamental groups in these dimensions.
He also proved that there are no isospectral non-isometric spherical space forms in dimension~$3$ (see~\cite[Thm.~1]{Ikeda80_3-dimI}), nor in dimension~$5$ with non-cyclic fundamental groups (see~\cite[Thm.~3.9]{Ikeda80_3-dimII}).
In addition, there are no examples addressing Question~\ref{question} in even dimension~$2d$, since $S^{2d}$ and the real projective space $\mathbb{P}^{2d}(\mathbb{R})$ are the only $2d$-dimensional spherical space forms, and they are obviously not isospectral.

Our first main theorem provides an affirmative answer to Question~\ref{question}.

\begin{theorem}\label{thm:main2-contraejemplos}
There exist infinitely many pairs of isospectral spherical space forms with non-isomorphic fundamental groups.
\end{theorem}

The proof of Theorem~\ref{thm:main2-contraejemplos} follows from the explicit construction provided in Theorem~\ref{thm:contraejemplos}. 
In particular, the dimension of such pairs is unbounded.

Subsection~\ref{subsec:ejemploscomputacionales} includes an algorithm to search for isospectral spherical space forms with non-isomorphic fundamental groups.
The results displayed in Table~\ref{table} show that the construction in Theorem~\ref{thm:contraejemplos} is quite general, since all the examples found are covered by that construction.

\begin{remark}
Among manifolds of constant sectional curvature, the positively curved case remained the only one for which no examples of isospectral manifolds with non-isomorphic fundamental groups were known.
In the negatively curved setting, Vign{\'e}ras~\cite{Vigneras80} constructed pairs of hyperbolic manifolds, while in the flat case, Doyle and Rossetti~\cite{DoyleRossetti04_tetra-didi} constructed the pair known as Tetra and Didi.
\end{remark}

One can observe from the computational results in Table~\ref{table} that the counterexamples to the audibility of the fundamental group (among spherical space forms) are rather special.
This naturally leads to the problem of finding conditions under which the fundamental group of a spherical space form can be heard.

Any group of Type~I is generated by elements $A$ and $B$ satisfying $A^m = B^n = e$ and $BAB^{-1} = A^r$, for integers $m,n,r$ such that
$
\gcd\big((r-1)n,m\big)=1
$ and 
$
r^n \equiv 1 \pmod m
$
(see Lemma~\ref{lem:Burnside1}).
We denote such a group by $\Gamma_d(m,n,r)$, where $d$ denotes the order of the class of $r$ in $\mathbb{Z}_m^\times$.
It turns out that all irreducible real representations of a fixed point free Type~I group $\Gamma_d(m,n,r)$ have dimension $2d$.
Moreover,
$
\Gamma_{d_1}(m_1,n_1,r_1)\simeq \Gamma_{d_2}(m_2,n_2,r_2)
$
if and only if $m_1=m_2$, $n_1=n_2$, $d_1=d_2$, and $r_1 \equiv r_2^c \pmod m$ for some $c\in\mathbb{Z}$ (see Proposition~\ref{prop:TypeI-uptoisomorphism}).

The next result shows that several features of fundamental groups of Type~I are audible.

\begin{theorem}\label{thm:main1}
Let $M_1$ and $M_2$ be isospectral spherical space forms with fundamental groups $G_1$ and $G_2$ respectively.
If $G_1$ is of Type~I, say $G_1 \simeq \Gamma_d(m,n,r_1)$, then $G_2$ is also of Type~I, with $G_2 \simeq \Gamma_d(m,n,r_2)$, and
\[
\gcd(r_1^c - 1, m) = \gcd(r_2^c - 1, m)
\qquad\text{for every positive divisor $c$ of $d$.}
\]
\end{theorem}

A simple corollary of this result, together with Remark~\ref{rem:Z_m*cyclic}, describes a situation in which the fundamental group is audible.

\begin{corollary}
Let $M$ be a spherical space form with fundamental group $G \simeq \Gamma_d(m,n,r)$ of Type~I. 
If $m=p^k$ or $m=2p^k$, for some odd prime $p$ and $k\in\mathbb{N}$, then any spherical space form isospectral to $M$ has fundamental group isomorphic to $G$.
\end{corollary}

The article is organized as follows.
Section~\ref{sec:preliminaries} reviews some facts on spherical space forms and groups of Type~I.
In Section~\ref{sec:main1}, we prove Theorem~\ref{thm:main1}.
The counterexamples to the audibility of the fundamental group of a spherical space form are presented in Section~\ref{sec:contraejemplos}, where Theorem~\ref{thm:main2-contraejemplos} is proved.
The article concludes with some open questions in the final section.

\subsection*{Acknowledgments}

The authors would like to thank the anonymous referee for a very careful reading of the manuscript and for numerous valuable suggestions and comments, all of which helped to improve the presentation of the paper.

\section{Preliminaries}\label{sec:preliminaries}

In this section we introduce the manifolds of interest for this article: 
spherical space forms with fundamental group of Type I. 
We refer the reader to \cite[Ch.~5--7]{Wolf-book} for a fuller treatment, and also to \cite[\S3--5]{Wolf01} for a summarized version.

\subsection{Spherical space forms}
A \emph{spherical space form} is a complete Riemannian manifold with constant positive sectional curvature, which we assume equal to 1. 
By the Killing-Hopf Theorem, we have that such a manifold is isometric to a quotient $S^n/G$, where $S^n$ is the unit sphere centered at the origin in $\R^{n+1}$ and $G$ is a discrete subgroup of the isometry group $\text{Iso}(S^n)\simeq \Ot(n+1)$ acting freely on $S^n$.
The latter condition is equivalent to requiring that $1$ is not an eigenvalue of the matrix $g$ for any $g \in G \smallsetminus \{\Id\}$.
Furthermore, the fundamental group of a spherical space form $S^n/G$ is isomorphic to $G$. 

Even dimensional spherical space forms are only the round spheres and the real projective spaces. When $G$ is cyclic and $n$ is odd, $S^n/G$ is called a \emph{lens space}. 
In this article we focus on odd-dimensional spherical space forms with non-cyclic fundamental groups.

Let $\Gamma$ be a finite group. 
A \emph{fixed point free representation} of $\Gamma$ is a faithful representation $\rho:\Gamma\to\Ot(q+1)$ such that $\rho(\Gamma)$ acts on the unit sphere $S^q$ without fixed points. An abstract finite group is called \emph{fixed point free} if it admits a fixed point free representation. 
If $\Gamma$ is a fixed point free group and $\rho:\Gamma\to \Ot(q+1)$ is a fixed point free representation of $\Gamma$ of dimension $q+1$, then $S^q/\rho(\Gamma)$ is a spherical space form. 
Of course, any spherical space forms can be written in this form. 

Isometric spherical space forms are in particular homeomorphic, therefore their fundamental groups are isomorphic. 
Two $q$-dimensional spherical space forms $S^q/\rho_1(\Gamma)$ and $S^q/\rho_2(\Gamma)$ are isometric if and only if there is an automorphism $\psi$ of $\Gamma$ such that $\rho_1$ is equivalent to $\rho_2\circ\psi$. 

Joseph Wolf, in his famous book \textit{Spaces of constant curvature} (\cite{Wolf-book}), classified fixed point free groups into six different types (I--VI). 
Also, for any given fixed point free group, he classified all possible fixed point free representations unless of isometry. 
We will focus in the Type I groups and their representations.

\subsection{Type I groups}\label{subsec:typeIgroups}

Following \cite{Wolf-book}, we call a finite group of \emph{Type~I} if all its Sylow subgroups are cyclic. 
Burnside proved the following two results that characterizes fixed point free Type~I groups (see \cite[Thm.~5.4.1]{Wolf-book}). 

\begin{lemma}\label{lem:Burnside1}
Let $m,n$ and $r$ be positive integers satisfying that $\gcd((r-1)n,m)=1$ and $r^n\equiv 1\pmod m$. 
Then, the group $\Gamma$ generated by elements 
$A,B$ satisfying
\begin{equation}\label{eq:TypeIrelations}
A^m=B^n=1
\qquad\text{and}\qquad 
BAB^{-1}=A^r, 
\end{equation}
has order $mn$ and is of Type~I. 
Conversely, any group of Type~I is isomorphic as one as above. 
\end{lemma}

We denote by $\Gamma_d(m,n,r)$ the Type I group as in Lemma~\ref{lem:Burnside1}, where 
\begin{equation*}
\text{$d$ denotes the order of the class of $r$ in $\Z_m^\times$.}
\end{equation*}
Here, $\Z_m^\times$ stands for the multiplicative group of integers modulo $m$. 
Note that $d$ divides $n$ since $r^n\equiv 1\pmod m$.

\begin{lemma}\label{lem:Burnside2}
A Type~I group $\Gamma_d(m,n,r)$ is fixed point free if and only if $n/d$ is divisible by any prime divisor of $d$.
\end{lemma}

The automorphism group of a Type I group $\Gamma_d(m,n,r)$ is given by (see \cite[Thm.~5.5.6]{Wolf-book}) 
\begin{equation}\label{eq:Aut(Gamma)}
\Aut(\Gamma_d(m,n,r)) = \{\psi_{s,t,u} : \gcd(s,m)=1=\gcd(t,n),\; t \equiv 1 \pmod d\},
\end{equation}
where $\psi_{s,t,u}$ is determined by $\psi_{s,t,u}(A)=A^s$ and $\psi_{s,t,u}(B)=B^tA^u$.

\begin{remark}\label{rem:Gamma_d(m,n,r)}
It is easy to see that any element in a Type I group $\Gamma_{d}(m,n,r)$ can be uniquely written as $A^aB^b$ with $0\leq a< m$ and $0\leq b<n$. 
In particular, its cardinality is equal to $mn$. 
Furthermore, 
\begin{equation*}
\Gamma_d (m,n,r) \text{ is cyclic }
\iff 
m=1
\iff
r=1
\iff 
d=1.
\end{equation*}
Note also that $m$ is necessarily odd because $\gcd(m,r)=\gcd(m,r-1)=1$. 
\end{remark}

\begin{lemma}\label{lem:(A^aB^b)^k}
Let $a,b$ be integer numbers with $b$ positive. 
On a Type~I group $\Gamma_d(m,n,r)$ with generators $A,B$ satisfying \eqref{eq:TypeIrelations}, one has that $B^bA^a=A^{ar^b}B^b$ and 
$$
(A^aB^b)^{k}= A^{a(1+r^b+r^{2b}+\dots+r^{(k-1)b} )}B^{kb}
\qquad\forall\, 
k\in\N. 
$$
\end{lemma}

\begin{proof}
The first assertion follows from $BAB^{-1}=A^r$. 
We proceed by induction on $k$ for the second.  
The case $k=1$ is immediate.  
Assume the formula holds for some $k\geq 1$. 
Hence
\begin{equation*}
\begin{aligned}
(A^aB^b)^{k+1}  &
= (A^aB^b)(A^{a(1+r^b+\cdots+r^{(k-1)b})}B^{kb})  
= A^a\big( B^b A^{a(1+r^b+\cdots+r^{(k-1)b})}\big) B^{kb}
\\ &
= A^a\big(  A^{a(1+r^b+\cdots+r^{(k-1)b})r^b} B^b\big) B^{kb}
= A^{a(1+r^b+\cdots+r^{kb})} B^{(k+1)b}
,
\end{aligned}
\end{equation*}
and the proof is complete. 
\end{proof}

The following result appears as Proposition~2.3 in \cite{Ikeda80_3-dimII}, without proof or further references.
We include a proof here for the sake of completeness.
The authors learned on MathOverflow \cite{Mathoverflow} from Dave Benson's answer that the converse is a particular case of the main theorem in \cite{Basmaji}.
In fact, the proof of the converse given below is based on HenrikRüping's answer.

\begin{proposition}\label{prop:TypeI-uptoisomorphism}
Type I groups $\Gamma_{d}(m,n,r_1)$ and $\Gamma_{d}(m,n,r_2)$ are isomorphic if and only if there exists an integer $c$ such that $r_1\equiv r_2^c \pmod{m}$, which is equivalent to $\langle r_1\rangle=\langle r_2\rangle$ as subgroups of $\Z_m^\times$.
\end{proposition}

\begin{proof}
For each $i=1,2$, let $A_i,B_i$ denote the generators of $\Gamma_i:=\Gamma_d(m,n,r_i)$ satisfying \eqref{eq:TypeIrelations}. 

We first assume that $r_1 \equiv r_2^{c_0} \pmod{m}$ for some $c_0 \in \Z$.
Note that $c_0$ is relatively prime to $d$.  Dirichlet's theorem on arithmetic progressions ensures that there is a prime number $c$ satisfying $c\equiv c_0\pmod d$ and $c\nmid n$. 
It follows that $r_1 \equiv r_2^{c} \pmod{m}$.

We define a homomorphism $\varphi \colon \Gamma_1 \rightarrow \Gamma_2$ by setting $\varphi(A_1)=A_2$ and $\varphi(B_1)=B_2^c$.
This map is well defined since
$\varphi(A_1)^m = A_2^{m}=1$,
$\varphi(B_1)^n = (B_2^c)^n =1$, and
\begin{align*}
\varphi(B_1A_1B_1^{-1})
= \varphi(B_1)\varphi(A_1)\varphi(B_1)^{-1}
&= B_2^c A_2 B_2^{-c}
\\ & 
= A_2^{r_2^c}
	\qquad\text{(applying $B_2A_2B_2^{-1}=A_2^{r_2}$ $c$-times)}
\\ & 
= A_2^{r_1}
	\qquad\text{(since $r_1 \equiv r_2^c \pmod m$)}
\\ & 
=\varphi(A_1^{r_1}).
\end{align*}

It remains to show that $\varphi$ is surjective.
Since $c$ is coprime to $n$, there exists an inverse $c^*$ of $c$ modulo $n$ (i.e.\ $c^* \in \Z$ satisfying $cc^* \equiv 1 \pmod n$).
Since $\varphi(B_1^{c^*})=B_2^{cc^*}=B_2$, we conclude that the generators $A_2$ and $B_2$ lie in the image of $\varphi$, and hence $\varphi$ is surjective.

We now assume that $\varphi \colon \Gamma_1 \to \Gamma_2$ is an isomorphism of groups.
Remark~\ref{rem:elements-order|m} will show that 
\[
\{\gamma \in \Gamma_i : \text{the order of $\gamma$ divides $m$}\} = \langle A_i\rangle.
\]
It follows that $\varphi$ maps $\langle A_1\rangle$ onto $\langle A_2\rangle$, so $\varphi(A_1)=A_2^a$ for some integer $a$ relatively prime to $m$.
The strategy is to compose $\varphi$ with automorphisms of $\Gamma_2$ (see \eqref{eq:Aut(Gamma)}) to get an adequate isomorphism between $\Gamma_1$ and $\Gamma_2$.

Let $a^* \in \Z$ be an inverse of $a$ modulo $m$.
The isomorphism $\psi_{a^*,1,0}\circ\varphi \colon \Gamma_1 \to \Gamma_2$ maps $A_1$ to $A_2$.
Therefore, we may assume without loss of generality that $\varphi(A_1)=A_2$.
In this case, we have $\varphi(B_1)=A_2^sB_2^c$ for some integers $c,s$ with $\gcd(c,n)=1$.

One can prove that $\gcd(r_2^{c}-1,m)=1$ because $\gcd(c,n)=1$ (see Lemma~\ref{lem:ordenes}\eqref{item:gcd(r^j-1,m)=1}\eqref{item:gcd(j,d)=1=>gcd(r^j-1,m)=1} below replacing $j=c$).
Thus, the divisor $1+r_2+r_2^2+\dots+r_2^{c-1}$ of $r_2^c-1$ is also relatively prime to $m$.
Since $r_2$ is relatively prime to $m$ as well, there exists an integer $u$ satisfying
\begin{equation}\label{eq:u-isomorphism}
ur_2(1+r_2+\dots+r_2^{c-1}) \equiv -s \pmod m.
\end{equation}

Now, the isomorphism $\psi_{1,1,u}\circ\varphi \colon \Gamma_1 \to \Gamma_2$ maps $A_1$ to $A_2$, and
\begin{align*}
(\psi_{1,1,u}\circ\varphi)(B_1)
&= \psi_{1,1,u}(A_2^sB_2^c)
= A_2^s \big(B_2A_2^u\big)^c
\\ & 
= A_2^s \, A_2^{ur_2(1+r_2+\dots+r_2^{c-1})}\,B_2^c
	\qquad\text{(by Lemma~\ref{lem:(A^aB^b)^k})}
\\ & 
= B_2^c
	\qquad\text{(by \eqref{eq:u-isomorphism})}.
\end{align*}
Again, we may assume without loss of generality that $\varphi$ maps $A_1$ to $A_2$ and $B_1$ to $B_2^c$.

Finally, we compute
\begin{align*}
A_2^{r_1}
&= \varphi(A_1)^{r_1}
=\varphi(A_1^{r_1})
=\varphi(B_1A_1B_1^{-1})
\\ & 
=\varphi(B_1)\varphi(A_1)\varphi(B_1)^{-1}
=B_2^{c}A_2B_2^{-c}
\\ & 
= A_2^{r_2^{c}}B_2^c\, B_2^{-c}
	\qquad\text{(by Lemma~\ref{lem:(A^aB^b)^k})}
\\ & 
= A_2^{r_2^{c}}.
\end{align*}
Since $A_2$ has order $m$, it follows that $r_1 \equiv r_2^c \pmod m$, as claimed.
\end{proof}

\begin{remark}\label{rem:Z_m*cyclic}
If $m$ is a positive integer satisfying that $\Z_m^\times$ is cyclic, then two Type I groups of the form (if they exist) $\Gamma_{d}(m,n,r_1)$ and $\Gamma_{d}(m,n,r_2)$ are necessarily isomorphic. 
Indeed, one has that $\langle r_1\rangle=\langle r_2\rangle $ because $\Z_m^\times$ is cyclic, thus $r_1^c\equiv r_2\pmod m$ for some $c\in\Z$, and the isomorphism follows by  Proposition~\ref{prop:TypeI-uptoisomorphism}.

It is well known that $\Z_m^\times$ is cyclic if and only if $m=p^k$ or $m=2p^k$ for some odd prime integer $p$ (see e.g.\ \cite[Thm.~3, \S4.1]{IrelandRosen}).

\end{remark}

\section{Audible properties}\label{sec:main1}

This section is devoted to prove Theorem~\ref{thm:main1}. 
We first describe explicitly the set of orders of a Type I group, which is an spectral invariant (among spherical space forms) by the foundational work of Ikeda. 
We then use this explicit description to show some audible properties of the fundamental group of Type I of a spherical space form. 
At the end, we show that being of Type I is audible among spherical space forms.

\subsection{The set of orders}

For an arbitrary finite group $G$, let us denote by $\sigma(G)$ the set of orders of $G$, that is, 
\begin{equation}
\sigma(G)= \{k\in \mathbb{N}: \exists\, g \in G\text{ of order $k$}\}.
\end{equation}
The following audible properties (among spherical space forms) were obtained by Ikeda (see \cite[Cor.~2.4, 2.8]{Ikeda80_3-dimI}). 

\begin{proposition}\label{prop:invariantesIkeda}
If $S^q/G_1$ and $S^q/G_2$ are isospectral spherical space forms, then $|G_1|=|G_2|$ and $\sigma(G_1)=\sigma(G_2)$. 
\end{proposition}

Clearly, in the situation in Proposition~\ref{prop:invariantesIkeda}, $G_1$ is cyclic if and only if $G_2$ is cyclic.

The next goal is to give an explicit expression for $\sigma(\Gamma_d(m,n,r))$.

\begin{lemma}\label{lem:ordenes}
Let $\Gamma_d(m,n,r)$ be a Type I group. 
Let $j$ be an integer satisfying $0\leq j<d$. 
Write $c=\gcd(j,d)$ and $u=\gcd(r^j-1,m)$.  
We have the following:
\begin{enumerate}
\item \label{item:m/u|sum(r^j)^k}
$\frac{m}{u}$ divides $\sum\limits_{k=0}^{d-1} r^{kj}$.

\item \label{item:gcd(r^j-1,m)=1}
$\gcd(r^j-1,m)=1$ 
if and only if $\sum\limits_{k=0}^{d-1}r^{jk}\equiv 0\pmod m$.

\item \label{item:gcd(j,d)=1=>gcd(r^j-1,m)=1}
If $\gcd(j,d)=1$, then $\sum\limits_{k=0}^{d-1}r^{jk}\equiv 0\pmod m$.

\item \label{item:gcd(r^c-1,m)=gcd(r^j-1,m)}
$\gcd(r^j-1,m)= \gcd(r^{c}-1,m)$. 

\item \label{item:gcd(m/u,u)=1}
$\gcd\big(\frac{m}{u},u\big)=1$. 
\end{enumerate}
\end{lemma}

\begin{proof}
\eqref{item:m/u|sum(r^j)^k}
It follows immediately from
$ 
(r^j-1)\sum_{k=0}^{d-1}(r^{j})^k= r^{jd}-1\equiv 0 \pmod m.
$

\eqref{item:gcd(r^j-1,m)=1}
The forward implication follows by \eqref{item:m/u|sum(r^j)^k}. 
For the converse, 
\begin{equation}
\begin{aligned}
0&
\equiv \sum_{k=0}^{d-1}r^{jk}
\equiv d+\sum_{k=0}^{d-1} \big((r^{j})^k -1\big) 
\equiv d+\sum_{k=0}^{d-1} (r^j-1)\Big(\sum_{h=0}^{k-1} (r^{j})^h\Big) \pmod m
\end{aligned}
\end{equation}
forces $\gcd(r^j-1,m)\mid d$, thus 
$\gcd(r^j-1,m)\mid \gcd(d,m)$. 
However, $\gcd(d,m)=1$ since $d\mid n$ and $\gcd(m,n)=1$, which implies $\gcd(r^j-1,m)=1$. 

\eqref{item:gcd(j,d)=1=>gcd(r^j-1,m)=1}
If $\gcd(j,d)=1$, then $\langle r^j\rangle= \langle r\rangle$ in $(\Z/m\Z)^\times $, so
\begin{equation}
\begin{aligned}
\sum_{k=0}^{d-1} (r^j)^k&
\equiv \sum_{k=0}^{d-1} r^k 
\equiv (r^d-1)\, (r-1)^* 
\equiv 0 \pmod m,
\end{aligned}
\end{equation}
where $(r-1)^*$ is any inverse of $r-1$ modulo $m$. 
(Remember from Lemma~\ref{lem:Burnside1} that $\gcd(r-1,m)=1$.)

\eqref{item:gcd(r^c-1,m)=gcd(r^j-1,m)}
Since $r^j-1=(r^c)^{\frac{j}{c}}-1 =(r^c-1)\sum_{k=0}^{\frac{j}{c}-1}(r^c)^k$, we have that $r^c-1\mid r^j-1$, which gives $\gcd(r^c-1,m)\mid \gcd(r^j-1,m)$. 
Since $\langle r^j\rangle= \langle r^c\rangle$ in $\Z_m^\times $, we have that 
\begin{equation}\label{eqn:3.3}
\begin{aligned}
(r^c-1)\sum_{k=0}^{\frac{d}{c}-1} (r^j)^k
= (r^c-1)\sum_{k=0}^{\frac{d}{c}-1} (r^c)^k
= (r^c)^{\frac{d}{c}}-1
\equiv 0 \pmod m
.
\end{aligned}
\end{equation}
This implies that $\frac{m}{\gcd(r^c-1,m)} =\frac{m}{\gcd(r^j-1,m)} \frac{\gcd(r^j-1,m)}{\gcd(r^c-1,m)}$ divides
$$
\sum_{k=0}^{\frac{d}{c}-1} (r^j)^k
=\frac{d}{c}+\sum_{k=0}^{\frac{d}{c}-1} \big( (r^j)^k-1 \big)
=\frac{d}{c}+\sum_{k=0}^{\frac{d}{c}-1} (r^j-1)\sum_{h=0}^{k-1}(r^j)^h.
$$
Hence $\frac{\gcd(r^j-1,m)}{\gcd(r^c-1,m)}\mid \frac{d}{c}$, therefore $\frac{\gcd(r^j-1,m)}{\gcd(r^c-1,m)}\mid \gcd(\frac{d}{c},m)=1$, and the assertion follows.

\eqref{item:gcd(m/u,u)=1}
Suppose that $p$ is a prime divisor of $\gcd(\frac mu,u)$. 
Thus $p\mid \sum_{k=0}^{d-1}r^{jk}$ by \eqref{item:m/u|sum(r^j)^k} and
\begin{equation}
p\mid 
\sum_{k=0}^{d-1}r^{jk} - \sum_{k=0}^{d-1} (r^j-1) \sum_{h=0}^{k-1} (r^{j})^h
=\sum_{k=0}^{d-1}r^{jk} - \sum_{k=0}^{d-1} \big((r^j)^k-1\big)
=d. 
\end{equation}
It follows that $p\mid \gcd(d,u)$, which is a contradiction since $\gcd(d,u)\mid \gcd(d,m)=1$.
We thus conclude that $\gcd(u,\frac{m}{u})=1$. 
\end{proof}

The next result gives an explicit description of the set of orders of a Type I group.
The case when $d$ is prime was proved in \cite[Cor.~3.6]{Ikeda80_3-dimII}.

\begin{theorem}\label{thm:sigma}
The set of orders of a Type I group $\Gamma_d (m,n,r)$ is given by 
$$
\sigma\big(\Gamma_d (m,n,r)\big)
= \bigcup_{c\mid d} \, \big\{k\in \N: k\mid \gcd(r^c-1,m)\tfrac{n}{c}\big\}
.
$$
\end{theorem}

\begin{proof}
Let $A$ and $B$ denote the generators of $\Gamma_d (m,n,r)$ satisfying \eqref{eq:TypeIrelations}. 

\begin{claim}\label{claim:order(A^aB^td+j)^un/c=1}
$(A^aB^{td+j})^{\frac{un}{c}}=1$ for all $a,t\in\Z$, $j\in\Z_{\geq0}$, where $u=\gcd(r^j-1,m)$ and $c=\gcd(j,d)$.
\end{claim}

\begin{proof}
\renewcommand{\qedsymbol}{$\blacksquare$}
Let $a,t,j$ be integers with $j\geq0$. 
By Lemma~\ref{lem:(A^aB^b)^k}, 
$
(A^aB^{td+j})^{\frac{n}{c}}
= A^{a\sum_{k=0}^{n/c-1} (r^j)^k} B^{\frac{(td+j)n}{c}}
= A^{a\sum_{k=0}^{n/c-1} (r^j)^k}
$
since $B$ has order $n$.
Thus 
$
(A^aB^{td+j})^{\frac{un}{c}}
= A^{au\sum_{k=0}^{n/c-1} (r^j)^k}
$.
It remains to show that $\sum_{k=0}^{n/c-1} (r^j)^k$ is divisible by $\frac{m}{u}$. 
Since $r^j$ has order $\frac{d}{c}$, we have that
\begin{align*}
\sum_{k=0}^{n/c-1} (r^j)^k 
\equiv  \frac{n}{d}\sum_{k=0}^{d/c-1} (r^j)^k \pmod m
\quad\text{and}\quad 
\sum_{k=0}^{d-1} (r^j)^k 
\equiv c \sum_{k=0}^{d/c-1} (r^j)^k  \pmod m
.
\end{align*}
Lemma~\ref{lem:ordenes}\eqref{item:m/u|sum(r^j)^k} ensures $\frac{m}{u}\mid \sum_{k=0}^{d-1} (r^j)^k$.
It follows from the identity at the right that $\frac{m}{u}\mid \sum_{k=0}^{d/c-1} (r^j)^k$ since $\gcd(\frac{m}{u},c)=1$. 
The assertion follows from the identity at the left. 
\end{proof}

\begin{claim}\label{claim:order(AB^c)=un/c}
Let $c$ be a positive divisor of $d$. 
Then $AB^c$ has order $\tfrac{un}{c}$, where $u=\gcd(r^c-1,m)$. 
\end{claim}

\begin{proof}
\renewcommand{\qedsymbol}{$\blacksquare$}
Let us denote by $h$ the order of $AB^c$. 
By taking $j=c$, $t=0$, and $a=1$ in Claim~\ref{claim:order(A^aB^td+j)^un/c=1} we obtain that $h\mid \tfrac{un}{c}$.  
Lemma~\ref{lem:(A^aB^b)^k} implies that
$ 
1=(AB^c)^h=A^{1+r^c+\dots+r^{c(h-1)}} B^{hc}.
$ 
Remark~\ref{rem:Gamma_d(m,n,r)} ensures 
\begin{align*}
n\mid hc
\quad\text{and}\quad 
\sum_{k=0}^{h-1}r^{ck}\equiv 0\pmod m.
\end{align*}
Since $\tfrac nc\mid h$ and $\gcd(u,\frac{n}{c})=1$ (because $u\mid m$ and $\gcd(m,n)=1$), it remains to show that $u\mid h$. 

Since $r^c$ has order $\frac{d}{c}$ in $\Z_m^\times$, we have that 
\begin{equation}
\begin{aligned}
0&\equiv \sum_{k=0}^{h-1} (r^{c})^k 
\equiv \frac{hc}{d} \sum_{k=0}^{d/c-1} (r^{c})^k  
\equiv \frac{hc}{d} \left(\sum_{k=0}^{d/c-1} \big((r^{c})^k-1\big) +\frac{d}{c}\right)\pmod m
.
\end{aligned}
\end{equation}
Since $u\mid m$ and $u\mid (r^{c})^k-1=(r^c-1)\sum_{h=0}^{k-1}(r^c)^h$ for any $k$, we conclude that $u\mid \frac{hc}{d}\cdot\frac{d}{c}=h$, as claimed. 
\end{proof}

For each divisor $c$ of $d$, we set $S_c=\{k\in\N : k\mid \gcd(r^c-1,m)\tfrac{n}{c}\}$. 
We want to prove that $\mathcal S:=\bigcup_{c\mid d} \mathcal S_c$ coincides with $\sigma(\Gamma_d(m,n,r))$.

Let $k$ be an element in $S_c$ for some $c\mid d$. 
If $\gcd(r^c-1,m)=1$, then $k\mid \frac nc$, so $k\mid n$. 
Since $B$ has order $n$, $B^{\frac{n}{k}}$ has order exactly $k$, thus $k\in \sigma(\Gamma_d(m,n,r))$. 
If $u:=\gcd(r^c-1,m)\neq 1$, since $AB^c$ has order $\frac{un}{c}$ by Claim~\ref{claim:order(AB^c)=un/c}, then $k\in \sigma(\Gamma_d(m,n,r))$ because $(AB^c)^{\frac{un}{ck}}$ has order $k$. 

We now prove that $\sigma\big(\Gamma_d(m,n,r)\big)\subset \mathcal S$. 
Picking $c=d$ in Claim~\ref{claim:order(AB^c)=un/c}, we obtain that $AB^d$ has order $\frac{mn}{d}$. 
Thus the order of any element in $\langle AB^d\rangle$ divides $\frac{mn}{d}$, so it lies in $\mathcal S_d\subset \mathcal S$. 

We next look at the elements in $\Gamma_d(m,n,r) \smallsetminus \langle AB^d\rangle$.

\begin{claim}\label{claim:<AB^d>=?}
$\langle AB^d\rangle= \{A^aB^{td}: 0\leq a<m,\; 0\leq t<\frac{n}{d}\}$.
\end{claim}

\begin{proof}
\renewcommand{\qedsymbol}{$\blacksquare$}
The forward inclusion is clear.
Let us prove the reverse inclusion.

Let $a,t$ be integers satisfying $0\leq a<m$ and $0\leq t<\frac{n}{d}$. 
Since $m$ is relatively prime to $n/d$, there is $q\in\N$ such that $mq\equiv t-a\pmod{\frac nd}$. 
By Lemma~\ref{lem:(A^aB^b)^k}, 
\begin{align*}
(AB^d)^{a+mq}&
= A^{a+mq} B^{(a+mq)d}
= A^{a} B^{td}
\end{align*}
since $(a+mq)d\equiv td\pmod n$, showing that $A^aB^{td}$ lies in $\langle AB^d\rangle$.
\end{proof}

Claim~\ref{claim:<AB^d>=?} implies that any element in $\Gamma_{d}(m,n,d)\smallsetminus\langle AB^d\rangle$ can be written as $A^aB^{td + j}$ for some integers $a,t,j$ satisfying $0\leq a<m$, $0\leq t<\frac{n}{d}$, and $0<j<d$. 
Write $c=\gcd(j,d)$. 
Now, Claim~\ref{claim:order(A^aB^td+j)^un/c=1} yields the order of $A^aB^{td + j}$ divides $\gcd(r^c-1,m)\frac{n}{c}$, so it lies in $\mathcal S_c$. 
\end{proof}

\begin{remark}\label{rem:elements-order|m}
Let $\Gamma_d(m,n,r)$ be a Type~I group with generators $A$ and $B$ satisfying \eqref{eq:TypeIrelations}.
We claim that
$$
\langle A\rangle
=
\{\gamma \in \Gamma_d(m,n,r): \text{the order of $\gamma$ divides $m$}\}.
$$
The forward inclusion is clear.
We prove the reverse inclusion using the proof of Theorem~\ref{thm:sigma}.
Let $\gamma\in \Gamma_d(m,n,r)$ be such that $k:=\op{ord}(\gamma)$ satisfies $k\mid m$.

Suppose, for contradiction, that
$
\gamma \notin \langle AB^d\rangle.
$
By Claim~\ref{claim:<AB^d>=?} in the proof of Theorem~\ref{thm:sigma}, we have
$
\gamma=A^aB^{td+j}
$
for some integers $a,t,j$ satisfying $0\leq a<m$, $0\leq t<\frac{n}{d}$, and $0<j<d$.
Applying Lemma~\ref{lem:(A^aB^b)^k}, we obtain
\[
1=(A^aB^{td+j})^k
=
A^{a(1+r^{td+j}+\cdots+r^{(k-1)(td+j)})} B^{k(td+j)},
\]
which implies that $n \mid k(td+j)$.
This is a contradiction since $0<td+j<n$ and $\gcd(n,k)=1$ (because $k\mid m$ and $\gcd(m,n)=1$).

The previous paragraph showed that $\gamma\in \langle AB^d\rangle$. 
Since $\langle AB^d\rangle$ is a cyclic group of order $\frac{mn}{d}$ (see Claim~\ref{claim:order(AB^c)=un/c} in the proof of Theorem~\ref{thm:sigma}), and furthermore $k=\op{ord}(\gamma)\mid m$ and $\gcd(m,\frac{n}{d})=1$, it follows immediately that $\gamma$ lies in $\langle (AB^d)^{\frac nd}\rangle = \langle A\rangle$, as claimed. 
\end{remark}

\subsection{Hearing parameters}

As a consequence of the explicit description for $\sigma\big(\Gamma_d(m,n,r)\big)$ in Theorem~\ref{thm:sigma}, we next deduce that the parameters $(m,n,d)$ are audible in the following sense.

\begin{proposition}\label{prop:m_1=m_2,n_1=n_2,d_1=d_2}
Let $M_1$ and $M_2$ be isospectral spherical space forms with fundamental groups $G_1$ and $G_2$ of Type I respectively, say $G_i\simeq \Gamma_{d_i}(m_i,n_i,r_i)$ for $i=1,2$. 
Then, 
$m:=m_1=m_2$, $n_1=n_2$, $d:=d_1=d_2$, and 
\begin{equation*}
\gcd(r_1^c-1,m)=\gcd(r_2^c-1,m)
\qquad\text{for every positive divisor $c$ of $d$}.
\end{equation*}
\end{proposition}

\begin{proof}
Since $M_1$ and $M_2$ are isospectral, 
$\sigma(G_1)= \sigma(G_2)$ and $|G_1|= |G_2|$ by Proposition~\ref{prop:invariantesIkeda}. 
We have that $|G_i|=m_in_i$ for $i=1,2$ by Lemma~\ref{lem:Burnside1}, thus $m_1n_1=m_2n_2$.  

If $d_1=1$, $G_1$ is cyclic (see Remark~\ref{rem:Gamma_d(m,n,r)}), thus $G_2$ is also cyclic because $\sigma(G_1)= \sigma(G_2)$ and $|G_1|=|G_2|$, concluding $d_2=1$. 
We assume from now on that $d_i\neq 1$ for $i=1,2$.

By Theorem~\ref{thm:sigma}, for $i=1,2$, 
\begin{equation}\label{eq:sigma(Gamma_i)}
\sigma(G_i)=
\bigcup_{c_i\mid d_i}
\big\{
	k\in\N : k\mid \gcd(r_i^{c_i}-1,m_i)\tfrac{n_i}{c_i}
\big\}
.
\end{equation}
Taking $c_1=d_1$ in \eqref{eq:sigma(Gamma_i)}, we get that $m_1\frac{n_1}{d_1} \in \sigma(\Gamma_1)$. 
Thus $m_1\frac{n_1}{d_1} \in\sigma(\Gamma_2)$, so there is a positive divisor $c_2$ of $d_2$ such that 
$m_1\frac{n_1}{d_1} 
\mid \gcd(r_2^{c_2}-1,m_2)\tfrac{n_2}{c_2}$.
Moreover, since the only element in $\sigma(\Gamma_1)$ divided by $m_1\frac{n_1}{d_1}$ is $m_1\frac{n_1}{d_1} $ itself, we have 
\begin{equation}\label{eq:m1n1_in_sigma(Gamma_2)}
m_1\tfrac{n_1}{d_1} 
= \gcd(r_2^{c_2}-1,m_2)\tfrac{n_2}{c_2}
.
\end{equation}
Analogously, there is a positive divisor $c_1$ of $d_1$ such that 
\begin{equation}\label{eq:m2n2_in_sigma(Gamma_1)}
m_2\tfrac{n_2}{d_2} 
= \gcd(r_1^{c_1}-1,m_1)\tfrac{n_1}{c_1}
.
\end{equation}
Set $u_i=\gcd(r^{c_i}_i-1,m_i)$ 
for $i=1,2$. 

Substituting $m_1n_1=m_2n_2$ at the left in \eqref{eq:m1n1_in_sigma(Gamma_2)} and \eqref{eq:m2n2_in_sigma(Gamma_1)}, we obtain that 
$m_2c_2=d_1u_2$ and $m_1c_1=d_2u_1$ respectively, or equivalently
\begin{align}\label{eq:d1=q2c2}
d_1&=\tfrac{m_2}{u_2}c_2, &
d_2&=\tfrac{m_1}{u_1}c_1. 
\end{align} 
This implies that $c_1\mid d_2$, thus $\gcd(c_1,\tfrac{m_2}{u_2})\mid \gcd(d_2,m_2)=1$, getting $\gcd(c_1,\tfrac{m_2}{u_2})=1$. 
Now, $c_1\mid d_1=\tfrac{m_2}{u_2}c_2$ forces $c_1\mid c_2$.
Proceeding analogously we obtain $c_2\mid c_1$, deducing $c_1=c_2$.

Substituting \eqref{eq:d1=q2c2} into $m_1\tfrac{n_1}{d_1}d_1=m_2\tfrac{n_2}{d_2}d_2$ we obtain that  
$
m_1\tfrac{n_1}{d_1} \tfrac{m_2}{u_2}c_2
=m_2 \tfrac{n_2}{d_2}\tfrac{m_1}{u_1}c_1
$, thus 
\begin{equation}\label{eq:u_1n_1'=u_2n_2'}
u_1\frac{n_1}{d_1}=u_2\frac{n_2}{d_2}. 
\end{equation}
Since $\gcd(\tfrac{m_2}{u_2},u_2)=1$ by Lemma~\ref{lem:ordenes}\eqref{item:gcd(m/u,u)=1} and  $\gcd(\tfrac{m_2}{u_2},\tfrac{n_2}{d_2})\mid \gcd(m_2,n_2)=1$, we obtain that
\begin{equation}\label{eq:1=gcd(q_2,u_1n_1')}
1=\gcd(\tfrac{m_2}{u_2},u_2\tfrac{n_2}{d_2})
=\gcd(\tfrac{m_2}{u_2},u_1\tfrac{n_1}{d_1}). 
\end{equation}

We claim that $\tfrac{m_2}{u_2}=1$. 
Indeed, if $p$ is a prime divisor of $\tfrac{m_2}{u_2}$,  $p\mid d_1$ by \eqref{eq:d1=q2c2}, thus $p\mid \tfrac{n_1}{d_1}$ by Lemma~\ref{lem:Burnside2} since $G_1$ is a fixed point free Type~I group, obtaining $p\mid \gcd(\tfrac{m_2}{u_2},u_1\tfrac{n_1}{d_1})$ which contradicts \eqref{eq:1=gcd(q_2,u_1n_1')}. 
We thus have $\tfrac{m_2}{u_2}=1$, and analogously $\tfrac{m_1}{u_1}=1$. 
It follows from $u_i=m_i$ that $c_i=d_i$ for any $i=1,2$ because $r_i$ has order $d_i$ in $\Z_{m_i}^\times$, deducing 
$$
d:=d_1=d_2
.
$$

We next prove that $m_1=m_2$ and $n_1=n_2$. 
Taking $c_1=1$ in \eqref{eq:sigma(Gamma_i)} we get $n_1\in\sigma(\Gamma_1)$ because $\gcd(r_1-1,m_1)=1$ (see Lemma~\ref{lem:Burnside1}). 
Since $n_1\in \sigma(\Gamma_2)$, there is a positive divisor $c$ of $d$ such that $n_1 \mid u\frac{n_2}{c}$, where $u=\gcd(r^{c}_2-1,m_2)$.
Hence
$$
c\tfrac{n_1}{d}  \mid u\tfrac{n_2}{d} = u\tfrac{m_1}{m_2}\tfrac{n_1}{d}
\quad\Longrightarrow\quad 
c  \mid u\tfrac{m_1}{m_2}
\quad\Longrightarrow\quad 
c \mid m_1
$$
since $\gcd(c,u)\mid\gcd(d,m_2)=1$.
If $u\neq 1$, then $c\neq 1$, which contradicts to $\gcd(m_1,d)=1$. 
Hence $u=1$, and thus 
$$
c\tfrac{n_1}{d}\mid \tfrac{n_2}{d}
=\tfrac{m_1}{m_2}\tfrac{n_1}{d}
\quad\Longrightarrow\quad 
m_2c\mid m_1
.
$$
We conclude that $m_2\mid m_1$ because $\gcd(c,m_1)=1$. 
Analogously, we obtain that $m_1\mid m_2$, so we conclude that $m_1=m_2$, and therefore $n_1=n_2$. 

We have thus proved so far that $m:=m_1=m_2$, $n:=n_1=n_2$, $d=d_1=d_2$. 
It only remains to show that 
$\gcd(r_1^c-1,m)=\gcd(r_2^c-1,m)$ for every positive divisor $c$ of $d$. 

We fix a positive divisor $c_1$ of $d$. 
Write $u_1=\gcd(r_1^{c_1}-1,m)$. 
Theorem~\ref{thm:sigma} gives $u_1\frac{n}{c_1}\in \sigma(\Gamma_1)=\sigma(\Gamma_2)$, thus there exists a positive divisor $c_2$ of $d$ such that 
\begin{equation}\label{eq:u'n/c'=un/c}
u_1\frac{n}{c_1}= u_2\frac{n}{c_2}
,
\end{equation}
where $u_2=\gcd(r_2^{c_2}-1,m)$. 
Since $\gcd(u_1,\frac{n}{c_2})=1$ (because $\gcd(u_1,\frac{n}{c_2})\mid \gcd(m,n)$ and $\gcd(m,n)=1$), we deduce that $u_1\mid u_2$. 
Analogously, we obtain that $u_2\mid u_1$, which implies $u_1=u_2$, and also $c_1=c_2$ by \eqref{eq:u'n/c'=un/c}.
We conclude that  $\gcd(r_1^{c}-1,m)=\gcd(r_2^{c}-1,m)$ for all $c\mid d$. 
\end{proof}

\subsection{Hearing the type}

Theorem~\ref{thm:main1} follows as a combination of Proposition~\ref{prop:m_1=m_2,n_1=n_2,d_1=d_2} and the next result. 

\begin{proposition}
Let $S^{q}/G_1$ and $S^{q}/G_2$ be isospectral spherical space forms. 
If $G_1$ is of Type I, then $G_2$ is also of Type I. 
\end{proposition}

\begin{proof}
We have that $G_1\simeq \Gamma_d(m,n,r)$ for some parameters $m,n,r,d$ as in Lemma~\ref{lem:Burnside1}.  
According to Wolf's classification of spherical space forms~\cite{Wolf-book}, the possible fundamental groups fall into six disjoint types, namely Types~I--VI.
The idea is to exploit the identities $|G_1|=|G_2|$ and $\sigma(G_1)=\sigma(G_2)$ valid by Proposition~\ref{prop:invariantesIkeda} to conclude that $G_2$ cannot be of Type II--VI. 

We assume first that $G_2$ is solvable, or equivalently, of Type I--IV (see \cite[Prop.~3.1]{Wolf01}). 
Suppose that $G_2$ is not of Type~I.  
Then, any $2$-Sylow subgroup of $G_2$ is isomorphic to a (non-cyclic) generalized quaternionic group of order $2^{u+1}$ for some $u\in\N$ (see \cite[\S6.1]{Wolf-book}). 
We have that $2^{u+1}$ divides $|G_2|=|G_1|=mn$, so $2^{u+1}\mid n$ because $m$ is necessarily odd (see Remark~\ref{rem:Gamma_d(m,n,r)}). 
It follows that there is an element in $G_1\simeq\Gamma_d(m,n,r)$ of order $2^{u+1}$, namely, the element corresponding to $B^{n/{2^{u+1}}}$, if $A,B$ denotes the generators of $\Gamma_d(m,n,r)$ satisfying \eqref{eq:TypeIrelations}.  
However, $\sigma(G_1)=\sigma(G_2)$ yields there is an element of order $2^{u+1}$ in $G_2$, which is a contradiction because any $2$-Sylow subgroup of $G_2$ is not cyclic. 

The group $G_2$ is not solvable if and only if $G_2$ is of Type~V--VI (see \cite[Prop.~3.2]{Wolf01}). 
Suppose $G_2$ is of Type~V, thus $G_2\simeq \Sigma\times\SL(2,\mathbb F_5)$, where $\Sigma$ is a fixed point free Type~I group satisfying $\gcd(|\Sigma|,30)=1$.
We have that $mn=|G_1|=|G_2|=120\, |\Sigma|$. 
Since $8=2^3$ divides $n$ (because $m$ is odd), $8\in\sigma(G_1)=\sigma(G_2)$, and this is a contradiction because any $2$-Sylow subgroup of $G_2$ is isomorphic to a $2$-Sylow of $\SL(2,\mathbb F_5)$ (because $|\Sigma|$ is odd), which is isomorphic to the quaternion group of order $8$. 

We end the proof assuming that $G_2$ is of Type~VI. 
We have that $G_2$ has a normal subgroup of index $2$ isomorphic to a group of Type V, say $\Sigma\times\SL(2,\mathbb F_5)$. 
Thus $mn=|G_1|=|G_2|=2\cdot 120\cdot |\Sigma|$, which gives in this case $16=2^4\mid n$.
Similarly as in the previous case, there is an element of order $16$ in $G_1$, and therefore also in $G_2$.
The last condition contradicts the fact the $2$-Sylow subgroups of $G_2$ are of order $16$ and not cyclic. 
\end{proof}

\section{Counterexamples}
\label{sec:contraejemplos}

This section is devoted to construct isospectral spherical space forms with non-isomorphic fundamental groups. 
After introducing fixed point free representations of a fixed point free Type~I group from \cite[\S5.5]{Wolf-book}, we give a construction in Theorem~\ref{thm:contraejemplos} which proves Theorem~\ref{thm:main2-contraejemplos}. 
The final subsection describes an algorithm to find such examples.

\subsection{Fixed point free representations}
For $\theta\in\R$, we write 
$
	R(\theta)=\left(\begin{smallmatrix}
		\cos(2\pi \theta) & \sin(2\pi \theta) \\
		-\sin(2\pi \theta) & \cos(2\pi \theta)
	\end{smallmatrix}\right).
$
The next result can be found in \cite[Thm.~5.5.10]{Wolf-book}. 

\begin{theorem} \label{thm:irreps}
Let $\Gamma_d(m,n,r)$ be a fixed point free Type I group (see Lemma~\ref{lem:Burnside1}) with cardinality $\geq3$, and let $A,B$ be its generators satisfying \eqref{eq:TypeIrelations}. 
Given integers $k$ and $l$ with $\gcd(k,m)=1=\gcd(l,n)$, let $\rho_{k,l}$ be the real representation of degree $2d$ of $\Gamma_d(m,n,r)$ determined by
\begin{align*}
\rho_{k,l}(A)&
=\begin{pmatrix}
		R(\frac{k}{m}) & 0 & \cdots & 0 \\
		0 & R(\frac{kr}{m}) & \cdots & 0 \\
		\vdots & \vdots & \ddots & \vdots \\
		0 & 0 & \cdots & R(\frac{kr^{d-1}}{m})
	\end{pmatrix}
	,&
\rho_{k,l}(B)&
=\begin{pmatrix}
		0&0& 1 &   \\
		\vdots&  \vdots & & \ddots & \\
		0&0&&&1\\
		& & 0  &  \cdots & 0 \\
		\multicolumn{2}{c}{\text{\raisebox{1.5ex}[0pt]{$R(\tfrac{l}{n/d})$}}}   & 0 & \cdots & 0
	\end{pmatrix}.
\end{align*}
Then, these representations are fixed point free and irreducible and, $\rho_{k,l}$ and $\rho_{k',l'}$ are equivalent if and only if there are  $\epsilon\in\{\pm1\}$ and $c\in\{0,\dots,d-1\}$ such that 
\begin{equation}\label{eq:irrep-equivalentes}
k'\equiv \epsilon kr^c\pmod m
\qquad\text{ and }\qquad
l'\equiv \epsilon l\pmod{n/d} 
.
\end{equation}
Moreover, any real representation of $\Gamma_d(m,n,r)$ is fixed point free if and only if it is equivalent to a sum of representations $\rho_{k,l}$ as above. 	
\end{theorem}

A spherical space form is called \emph{irreducible} when the defining representation of the fixed point free group is irreducible. 
The next observation ensures that $k$ above can be always assumed equal to $1$ for irreducible spherical space forms of Type I. 

\begin{remark}\label{rem:k=1}
We continue under the notation of Theorem~\ref{thm:irreps}. 
Let $s,t,u$ be integers satisfying $\gcd(s,m)=1=\gcd(t,n)$ and $t\equiv 1\pmod d$.
The automorphism $\psi_{s,t,u}$ of $\Gamma_{d}(m,n,r)$ defined in \eqref{eq:Aut(Gamma)} satisfies that $\rho_{k,l}\circ\psi_{s,t,u}$ is equivalent to $\rho_{sk,tl}$ (see \cite[Thm.~5.5.6]{Wolf-book}).
In particular $S^{2d-1}/\rho_{k,l}(\Gamma_d(m,n,r))$ and $S^{2d-1}/\rho_{sk,tl}(\Gamma_d(m,n,r))$ are isometric. 

For this reason, taking $s$ to be the inverse of $k$ modulo $m$, we may assume that $k=1$.
\end{remark}

\subsection{Isospectral construction}
The main goal of this subsection is to prove that two fixed point free Type I groups $\Gamma_1:=\Gamma_d(m,n,r_1)$ and $\Gamma_2:=\Gamma_d(m,n,r_2)$ satisfying $n=2d$ and $r_1r_2\equiv-1\pmod m$ induce pairs of isospectral spherical space forms.
Moreover, one can pick the parameters in such a way that the groups are not isomorphic. 

Given $A\in\Ot(q+1)$, we write $\Eig(A)$ for the multiset of $q+1$ complex eigenvalues of the matrix $A$. 

\begin{proposition}\label{prop:almost-conjugate}
Let $S^{q}/G_1$ and $S^{q}/G_2$ be two spherical space forms. 
If there is a bijection $\phi:G_1\to G_2$ satisfying $\Eig(g)=\Eig(\phi(g))$ for all $g\in G_1$, then $S^{q}/G_1$ and $S^{q}/G_2$ are strongly isospectral, in particular, $\Spec\big(S^{q}/G_1\big) = \Spec\big(S^{q}/G_2\big)$.
\end{proposition}

Two arbitrary subgroups $G_1$ and $G_2$ of $\Ot(q+1)$ satisfying the hypothesis in Proposition~\ref{prop:almost-conjugate} are called \emph{almost conjugate}, since the bijection $\phi$ preserves conjugacy classes in $\Ot(q+1)$.
The result above can be found in~\cite[Cor.~2.13]{Wolf01}; it extends Sunada's method as developed by DeTurck and Gordon~\cite{DeTurckGordon89} and was later interpreted in Lie-theoretic terms by Bérard~\cite{Berard93}.
Under the same hypotheses, Ikeda~\cite[Cor.~2.3]{Ikeda80_3-dimI} had previously proved isospectrality.

We are now in position to prove Theorem~\ref{thm:main1}, which immediately follows from the next result. 
As a convention, for each $i=1,2$, we denote by $\rho_{k,l}^{(i)}$ the irreducible fixed point free representation of $\Gamma_i$ defined in Theorem~\ref{thm:irreps}.

\begin{theorem}\label{thm:contraejemplos}
Let $\Gamma_1=\Gamma_{d}(m,n,r_1)$ and $\Gamma_2=\Gamma_{d}(m,n,r_2)$ be non-cyclic fixed point free Type I groups. 
If $n=2d$ and $r_1r_2\equiv-1\pmod m$, then 
\begin{align*}
S^{2dp-1}/\big(\rho_{k_1,l_1}^{(1)}\oplus\dots\oplus \rho_{k_p,l_p}^{(1)}\big)(\Gamma_1) 
\quad\text{and}\quad
S^{2dp-1}/\big(\rho_{k_1,l_1}^{(2)}\oplus\dots\oplus \rho_{k_p,l_p}^{(2)}\big)(\Gamma_2)
\end{align*}
are strongly isospectral for all integers $k_1,l_1,\dots,k_p,l_p$ satisfying $\gcd(k_j,m)=1=\gcd(l_j,n)$ for all $j$. 
\end{theorem}

\begin{proof}
The plan is to use Proposition~\ref{prop:almost-conjugate} with the bijection between $\Gamma_1$ and $\Gamma_2$ given by $A_1^aB_1^b\mapsto A_2^aB_2^b$ for any $a,b\in\Z$, where $A_i,B_i$ stand for the generators of $\Gamma_i$ satisfying \eqref{eq:TypeIrelations}, for $i=1,2$. 

Clearly 
$\Eig\big((\rho_{k_1,l_1}^{(i)}\oplus\dots\oplus \rho_{k_p,l_p}^{(i)})(\gamma)\big)
=\bigcup_{j=1}^p \Eig\big(\rho_{k_j,l_j}^{(i)}(\gamma)\big)$ for any $\gamma\in \Gamma_i$. 
Hence, it is sufficient to show that 
$
\Eig\big(\rho_{k_j,l_j}^{(1)}(A_1^aB_1^b)\big) 
= \Eig\big(\rho_{k_j,l_j}^{(2)}(A_2^aB_2^b)\big) 
$
for all $1\leq j\leq p$. 
Consequently, there is no loss of generality in assuming $p=1$, as well as $k_1=1$ by Remark~\ref{rem:k=1} and $l_1=1$ by \eqref{eq:irrep-equivalentes} because $n/d=2$. 

We need to show that 
$
\Eig\big(\rho_{1,1}^{(1)}(A_1^aB_1^b)\big) 
= \Eig\big(\rho_{1,1}^{(2)}(A_2^aB_2^b)\big) 
$
for all $a,b\in\Z$. 
The real representation $\rho_{1,1}^{(i)}$ of $\Gamma_i$ is of complex type, which means there is an irreducible complex representation $\pi_{1,1}^{(i)}$ of $\Gamma_i$ (explicitly defined in \cite[Thm.~5.5.6]{Wolf-book}) such that $\rho_{1,1}^{(i)}\simeq \pi_{1,1}^{(i)}\oplus \overline{\pi_{1,1}^{(i)}}$. 
Note that 
\begin{equation}\label{eq:Eig(rho_11^i(gamma))}
\Eig\big(\rho_{1,1}^{(i)}(\gamma)\big)
= \Eig\big(\pi_{1,1}^{(i)}(\gamma)\big)\cup \overline{\Eig\big(\pi_{1,1}^{(i)}(\gamma)\big)}
\qquad\forall\, \gamma\in \Gamma_i,
\end{equation}
where the bar means complex conjugation of every element in the multiset $\Eig\big(\pi_{1,1}^{(i)}(\gamma)\big)$.

In the sequel, we write $\xi_m=e^{2\pi\mi/m}$ for any $m\in\N$. 
Ikeda~\cite[(3.9)]{Ikeda83} determined explicitly the eigenvalues of $\pi_{1,1}^{(i)}(A_i^aB_i^b)$ as the roots of the polynomial
\begin{align*}
P_{a,b}^{(i)}(z)
:= \prod_{j=1}^{\gcd(d,b)} 
\left(
	z^{\frac{d}{\gcd(b,d)}}
	-\xi_m^{a\, \alpha^{(i)}(b) \, r_i^{j-1}}
	\xi_{n/d}^{\frac{b}{\gcd(d,b)}}
\right)\in\C[z]
,
\end{align*}
where 
$$
\alpha^{(i)}(c)=\sum_{h=1}^{\frac{d}{\gcd(d,c)}} r_i^{\gcd(d,c)h}
\qquad\text{for any $c\in\Z$}. 
$$ 
It follows that the eigenvalues of $\overline{\pi_{1,1}^{(i)}(A_i^aB_i^b)}$ are the roots of the polynomial 
$
\bar P_{a,b}^{(i)}(z):=\overline{P_{a,b}^{(i)}(\bar z)}\in\C[z],
$
whose $k$-th coefficient is the complex conjugate of the $k$-th coefficient of $P_{a,b}^{(i)}(z)$ for every $k$. 
Combining these facts with \eqref{eq:Eig(rho_11^i(gamma))}, the required assertion 
$
\Eig\big(\rho_{1,1}^{(1)}(A_1^aB_1^b)\big) 
= \Eig\big(\rho_{1,1}^{(2)}(A_2^aB_2^b)\big) 
$
is equivalent to 
\begin{equation*}
P_{a,b}^{(1)}(z)\bar P_{a,b}^{(1)}(z)
= 
P_{a,b}^{(2)}(z)\bar P_{a,b}^{(2)}(z)
.
\end{equation*}

Since $n = 2d$ and every prime divisor of $d$ divides $\frac{n}{d}=2$ by Lemma~\ref{lem:Burnside2}, it follows that $d$ is a power of $2$.

\setcounter{claim}{0}

\begin{claim}\label{claim:gcd}
$\gcd(r_1^{q}-1,m)=\gcd(r_2^{q}-1,m)$ for every positive divisor $q$ of $d$.
\end{claim}

\begin{proof}
	\renewcommand{\qedsymbol}{$\blacksquare$}
	The assertion is obvious for $q=1$ by the definition of a Type I group (i.e.\ $\gcd(r_i-1,m)=1$). 
	Let $q$ be a positive divisor of $d$ with $q\neq1$. 
	Note that $q$ is even because $d$ is a power of $2$.

	Since $r_1r_2\equiv -1\pmod m$, $r_1\equiv -r_2^{d-1}\pmod m$, thus $r_1^q\equiv r_2^{q(d-1)}\pmod m$.
	We conclude that
	\begin{align*}
		\gcd(r_2^q-1,m) &
		= \gcd(r_2^{q(d-1)}-1,m)
		\qquad
		\text{(by Lemma~\ref{lem:ordenes}\eqref{item:gcd(r^c-1,m)=gcd(r^j-1,m)})}
		\\ & 
		= \gcd(r_1^{q}-1,m),
	\end{align*}
	as claimed. 
\end{proof}

\begin{claim}\label{claim:alpha^i(b)}
$\alpha^{(1)}(c)\equiv \alpha^{(2)}(c)\pmod m$ for every $c\in\Z$. 
\end{claim}

\begin{proof}
\renewcommand{\qedsymbol}{$\blacksquare$}
If $\gcd(d,c)=1$, then $\alpha^{(i)}(c)\equiv 0\pmod m$ by Lemma~\ref{lem:ordenes}\eqref{item:gcd(j,d)=1=>gcd(r^j-1,m)=1} for any $i=1,2$. 

We suppose now that $\gcd(d,c)>1$. 
Note that $\gcd(d,c)$ is even because $d$ is a power of $2$. 

Similarly as in the proof of the previous claim, 
$r_1r_2\equiv -1\pmod m$ implies that
$r_1^{\gcd(d,c)}\equiv r_2^{\gcd(d,c)(d-1)}\pmod m$.
Hence
\begin{align*}
\alpha^{(2)}(c)
=\sum_{h=1}^{\frac{d}{\gcd(d,c)}} r_2^{\gcd(d,c)h}
\equiv \sum_{h=1}^{\frac{d}{\gcd(d,c)}} r_2^{\gcd(d,c)(d-1)h}
\equiv \sum_{h=1}^{\frac{d}{\gcd(d,c)}} r_1^{\gcd(d,c)h}
= \alpha^{(1)}(c)
\pmod m,
\end{align*}
as claimed. 
\end{proof}

Fix $a,b\in\Z$. 
Set $u=\gcd(r_1^{\gcd(d,b)}-1,m)=\gcd(r_2^{\gcd(d,b)}-1,m)$ (see Claim~\ref{claim:gcd}).
We have that
\begin{align*}
\left(r_i^{\gcd(d,b)}-1\right) \alpha^{(i)}(b)
&
= \left(r_i^{\gcd(d,b)}-1\right) \sum_{h=1}^{\frac{d}{\gcd(d,b)}} \big(r_i^{\gcd(d,b)}\big)^h
= r_i^{\gcd(d,b)}(r_i^{d}-1)&
\equiv 0\pmod m.
\end{align*}
We deduce that $\frac{m}{u}$ divides $\alpha^{(i)}(b)$, say $\alpha^{(i)}(b)=c_i\frac{m}{u}$ for some $c_i\in\Z$. 
Hence  
\begin{equation}\label{eq:P_ab}
\begin{aligned}
P_{a,b}^{(i)}(z) &
= \prod_{j=1}^{\gcd(d,b)} 
\left(
	z^{\frac{d}{\gcd(b,d)}}
	-\xi_m^{ac_i\, \frac{m}{u} \, r_i^{j-1}}
	\xi_{n/d}^{\frac{b}{\gcd(d,b)}}
\right)
= \prod_{j=1}^{\gcd(d,b)} 
\left(
	z^{\frac{d}{\gcd(b,d)}}
	-\xi_u^{ac_i\, r_i^{j-1}}
	\xi_{n/d}^{\frac{b}{\gcd(d,b)}}
\right)
,
\end{aligned}
\end{equation}
Claim~\ref{claim:alpha^i(b)} implies that $c_1\frac{m}{u}\equiv c_2\frac{m}{u} \pmod m$, thus $c_1\equiv c_2\pmod u$. 

Note that $\xi_{n/d}=-1$ because $n/d=2$. 
Since the term $\xi_{n/d}^{\frac{b}{\gcd(d,b)}}$ is real, we have that
\begin{equation}\label{eq:P_ab-conjugado}
	\begin{aligned}
\bar P_{a,b}^{(i)}(z) &
= \prod_{j=1}^{\gcd(d,b)} 
\left(
	z^{\frac{d}{\gcd(b,d)}}
	-\overline{
		\xi_u^{ac_i\, r_i^{j-1}}
		\xi_{n/d}^{\frac{b}{\gcd(d,b)}}
	}
\right)
= \prod_{j=1}^{\gcd(d,b)} 
\left(
	z^{\frac{d}{\gcd(b,d)}}
	-\xi_u^{-ac_i\, r_i^{j-1}}
	\xi_{n/d}^{\frac{b}{\gcd(d,b)}}
\right)
,
\end{aligned}
\end{equation}

Now, $r_1r_2\equiv -1\pmod m$ gives $r_1^{j-1}\equiv (-1)^{j-1}\, r_2^{d-(j-1)}\pmod m$ for all $1\leq j\leq d$. 
Moreover, 
$
r_1^{j-1}\equiv (-1)^{j-1}\, r_2^{\gcd(b,d)-(j-1)}\pmod u
$
for all $1\leq j\leq \gcd(d,b)$
because $u\mid m$ and $r_i^{\gcd(b,d)}\equiv 1\pmod u$. 
Combining this with $c_1\equiv c_2\pmod u$, we get 
\begin{align*}
ac_1\, r_1^{j-1}&\equiv 
(-1)^{j-1}\, ac_2\, r_2^{\gcd(b,d) -(j-1)}\pmod u
\qquad\forall \,1\leq j\leq \gcd(b,d). 
\end{align*}
We have thus shown, for any $1\leq j\leq \gcd(b,d)$, that
\begin{align*}
\xi_u^{\pm ac_1\, r_1^{j-1}}
= 
\begin{cases}
\xi_u^{\pm ac_2\, r_2^{j-1}}
	&\text{ if $j-1$ is even},\\
\xi_u^{\mp ac_2\, r_2^{j-1}}
&\text{ if $j-1$ is odd}.
\end{cases}
\end{align*}
It follows immediately from \eqref{eq:P_ab} and \eqref{eq:P_ab-conjugado} that 
$
P_{a,b}^{(1)}(z)\bar P_{a,b}^{(1)}(z)
= 
P_{a,b}^{(2)}(z)\bar P_{a,b}^{(2)}(z)
$, and the proof is complete.
\end{proof}

The next remark shows that $15$ is the smallest dimension of a pair of isospectral spherical space forms with non-isomorphic fundamental groups produced by  Theorem~\ref{thm:contraejemplos}.

\begin{remark}\label{rem:smallest-dim}
As mentioned in the proof of Theorem~\ref{thm:contraejemplos}, the hypothesis $n = 2d$ implies that both $d$ and $n$ are powers of $2$, as a consequence of Lemma~\ref{lem:Burnside2}.
The case $d=1$ is excluded because $\Gamma_d(m,n,d)$ is assumed to be non-cyclic (see Remark~\ref{rem:Gamma_d(m,n,r)}).
The case $d=2$ gives isomorphic groups (see Proposition~\ref{prop:d=2=>isomorphic} below). 
We next see that the choice $d=4$ also gives isomorphic groups. 

Under the assumption in Theorem~\ref{thm:contraejemplos}, assume that $d=4$, so $r_i^4\equiv1\pmod m$. 
From $r_1r_2\equiv -1\pmod m$, it follows that $r_1\equiv -r_2^3\pmod m$, so $\gcd(-r_2^3-1,m)=\gcd(r_1-1,m)=1$ (see Lemma~\ref{lem:Burnside1}). 
Since $-r_2$ is also relative prime to $m$, we have that 
$$
1=\gcd\big(-r_2(-r_2^3-1),m\big) = \gcd\big(r_2+1,m\big)
.
$$
Now, $r_2^4\equiv 1\pmod m$ implies $m\mid (r_2^4-1)=(r_2-1)(r_2+1)(r_2^2+1)$, deducing that $r_2^2\equiv -1\pmod m$. 
Hence 
\begin{align*}
r_1\equiv -r_2^3=(-r_2)r_2^2\equiv (-r_2)(-1)=r_2\pmod m,
\end{align*}
so $\Gamma_4(m,n,r_1)\simeq \Gamma_4(m,n,r_2)$ by Proposition~\ref{prop:TypeI-uptoisomorphism}. 

For $d=8$, we have that the Type~I groups
$\Gamma_8(85,16,2)$ and $\Gamma_8(85,16,42)$ satisfy the hypotheses in Theorem~\ref{thm:contraejemplos} and they are not isomorphic because $42 \not\equiv 2^c \pmod{85}$ for any $c \in \mathbb{Z}$.
Consequently, the isospectral pair
\begin{equation}\label{eq:smallest-example}
S^{15}/\rho_{1,1}^{(1)}\big(\Gamma_8(85,16,2)\big),
\quad
S^{15}/\rho_{1,1}^{(1)}\big(\Gamma_8(85,16,42)\big)
\end{equation}
has the smallest possible dimension among the isospectral pairs of spherical space forms with non-isomorphic fundamental groups produced by Theorem~\ref{thm:contraejemplos}.
Moreover, the computational results of the next subsection show that the pair~\eqref{eq:smallest-example} has the largest volume (equivalently, the smallest fundamental groups) among all isospectral spherical space forms with non-isomorphic fundamental groups arising from Theorem~\ref{thm:contraejemplos}.
\end{remark}

\subsection{Computational results}\label{subsec:ejemploscomputacionales}

We next explain the computational results which reveal that the construction in Theorem~\ref{thm:contraejemplos} is quite general.

\begin{table}
\caption{
Isospectral spherical space forms with non-cyclic, non-isomorphic fundamental groups. 
}
\label{table}
\small
$
\begin{array}[t]{ccccl}
N&m&n&d& [r_1,r_2]\\ \hline
1360 & 85 & 16 & 8 & [2, 42] \\
2720 & 85 & 32 & 16 & [3, 12] \\
3280 & 205 & 16 & 8 & [3, 68] \\
3536 & 221 & 16 & 8 & [8, 138] \\
5840 & 365 & 16 & 8 & [22, 168] \\
6800 & 425 & 16 & 8 & [32, 168] \\
7072 & 221 & 32 & 16 & [5, 44] \\
7120 & 445 & 16 & 8 & [12, 37] \\
7760 & 485 & 16 & 8 & [33, 227] \\
7888 & 493 & 16 & 8 & [70, 104] \\
8528 & 533 & 16 & 8 & [44, 96] \\
9040 & 565 & 16 & 8 & [18, 357] \\
10064 & 629 & 16 & 8 & [43, 117] \\
10960 & 685 & 16 & 8 & [127, 452] \\
11152 & 697 & 16 & 8 & [9, 196] \\
11152 & 697 & 16 & 8 & [38, 55] \\
13600 & 425 & 32 & 16 & [7, 82] \\
14416 & 901 & 16 & 8 & [76, 348] \\
15184 & 949 & 16 & 8 & [83, 229] \\
15440 & 965 & 16 & 8 & [43, 588] \\
15520 & 485 & 32 & 16 & [8, 202] \\
15776 & 493 & 32 & 16 & [12, 41] \\
16400 & 1025 & 16 & 8 & [68, 232] \\
16592 & 1037 & 16 & 8 & [111, 172] \\
17680 & 1105 & 16 & 8 & [8, 138]\\
17680 & 1105 & 16 & 8 & [83, 213] \\
18080 & 565 & 32 & 16 & [42, 48] \\
18512 & 1157 & 16 & 8 & [190, 749] \\
\end{array}
$
\qquad
$
\begin{array}[t]{ccccl}
N&m&n&d& [r_1,r_2] \\ \hline
18640 & 1165 & 16 & 8 & [12, 97] \\
19024 & 1189 & 16 & 8 & [191, 249] \\
19280 & 1205 & 16 & 8 & [8, 693] \\
19856 & 1241 & 16 & 8 & [100, 246]\\
19856 & 1241 & 16 & 8 & [302, 387] \\
20128 & 629 & 32 & 16 & [6, 80] \\
20176 & 1261 & 16 & 8 & [47, 629] \\
20560 & 1285 & 16 & 8 & [193, 253] \\
22304 & 697 & 32 & 16 & [3, 232]\\
22304 & 697 & 32 & 16 & [14, 44]\\
22304 & 697 & 32 & 16 & [73, 173] \\
22480 & 1405 & 16 & 8 & [192, 473] \\
23120 & 1445 & 16 & 8 & [423, 468] \\
23504 & 1469 & 16 & 8 & [18, 44] \\
24208 & 1513 & 16 & 8 & [144, 212]\\
24208 & 1513 & 16 & 8 & [166, 319] \\
24272 & 1517 & 16 & 8 & [68, 191] \\
25040 & 1565 & 16 & 8 & [188, 308] \\
26384 & 1649 & 16 & 8 & [47, 64]\\
26384 & 1649 & 16 & 8 & [172, 366] \\
26960 & 1685 & 16 & 8 & [252, 1122] \\
27472 & 1717 & 16 & 8 & [111, 818] \\
28240 & 1765 & 16 & 8 & [237, 943] \\
28496 & 1781 & 16 & 8 & [96, 421] \\
28832 & 901 & 32 & 16 & [23, 182] \\
29200 & 1825 & 16 & 8 & [168, 793] \\
29648 & 1853 & 16 & 8 & [76, 185] \\
\end{array}
$

\smallskip

For $m,n,d$ in some entry, $N=mn$ is the order of the group Type I group $\Gamma_{d}(m,n,r_i)$. \\
Each pair $[r_1,r_2]$ at the last column means that  $S^{2d-1}/\rho_{1,1}(\Gamma_{d}(m,n,r_1))$ and  $S^{2d-1}/\rho_{1,1}(\Gamma_{d}(m,n,r_2))$ are isospectral, and 
$\Gamma_{d}(m,n,r_1)\not\simeq \Gamma_{d}(m,n,r_2)$. 
\end{table}

We denote by $\mathcal P_{q,k}$ the space of homogeneous complex polynomials in $q+1$ variables of degree $k$, and by $\mathcal H_{q,k}$ its subspace given by harmonic elements. 
For any finite subgroup $G$ of $\Ot(q+1)$, we associate the generating function given by 
\begin{align*}
F_G(z):=
\sum_{k\geq0} \dim\mathcal H_{q,k}^G\, z^k.
\end{align*}

Ikeda noted (see \cite[Prop.~2.1]{Ikeda80_3-dimI}), for two spherical space forms $S^{q}/G_1$, $S^{q}/G_2$, that 
\begin{equation}\label{eq:encode}
\Spec(S^{q}/G_1)=\Spec(S^{q}/G_2)
\quad\Longleftrightarrow\quad 
F_{G_1}(z)=F_{G_2}(z)
.
\end{equation}
Also, for any finite subgroup $G$ of $\Ot(q+1)$, he obtained the expression (see \cite[Thm.~2.2]{Ikeda80_3-dimI})
\begin{equation}\label{eq:F_G(z)}
F_G(z)=\frac{1-z^2}{|G|} \sum_{g\in G}\frac{1}{\det(\Id_{q+1}-gz)},
\end{equation}
where $\det(\Id_{q+1}-gz)=\prod_{\lambda\in \Eig(g)} (1-\lambda z)$. 

\begin{remark}\label{rem:det(z-g)}
Setting $q=2d-1$, the term $\det(z\Id_{2d}-g)= \det(\Id_{2d}-gz)$ for $g=\rho_{1,1}(A^aB^b)$ is precisely the product of the polynomials $P_{a,b}(z)\bar P_{a,b}(z)$ defined inside the proof of Theorem~\ref{thm:contraejemplos}. 
\end{remark}

The generating function introduced above has been a fruitful tool for finding, or establishing the absence of, isospectral spherical space forms, and has been used, for instance, in \cite{Ikeda80_3-dimI, Ikeda80_3-dimII, Ikeda80_isosp-lens, Ikeda83} and \cite{Wolf-book} (see also \cite[\S3, Thm.~5.11, \S6]{LMR-SaoPaulo}).
It was also essential in the computational study of $p$-isospectral lens spaces carried out in \cite{Lauret-computationalstudy}.
The goal of this subsection is to adapt the codes from \cite{Lauret-computationalstudy}, developed for cyclic fundamental groups, to non-cyclic fundamental groups of Type~I.

To simplify the search, we will only consider spherical space forms of the form $$S^{2d-1}/ \rho_{1,1}\big(\Gamma_d(m,n,r))\big).$$
For $s\in\Z$, let $[s]_m$ denote the only integer satisfying $0\leq [s]_m<m$ and $[s]_m\equiv s\pmod m$. 

\begin{algorithm} \label{algorithm}
Searching for non-isomorphic, non-cyclic, fixed point free Type I groups 
$\Gamma_d(m,n,r_1)$ and $\Gamma_d(m,n,r_2)$ 
such that 
$S^{2d-1}/\rho_{1,1}(\Gamma_d(m,n,r_1))$ and $(S^{2d-1}/\rho_{1,1}(\Gamma_d(m,n,r_2))$ are isospectral.
\begin{enumerate}
\item Run $N$ over the positive integer numbers. 

\item Find all non-cyclic fixed point free Type I groups of order $N$ up to isomorphism, i.e.\ find all positive integers $m,n,d,r$ satisfying that
	\begin{itemize}
	\item $N=mn$,
	\item $\gcd((r-1)n,m)=1$, 
	\item $d\mid n$, $d\neq 1$, 
	\item the class of $r$ in $\Z_m^\times$ has order $d$, 
	\item $p\mid \frac{n}{d}$ for all prime divisor $p$ of $d$,
	\item $0\leq r\leq [r^c]_m$ for all $c\in\Z$ with $\gcd(c,d)=1$.
	\end{itemize}
The last condition ensures that the resulting Type I groups are not isomorphic by Theorem~\ref{prop:TypeI-uptoisomorphism}. 

\item \label{item:F_G(z)}
For each $\Gamma:=\Gamma_d(m,n,r)$ from the previous item, compute the generating function $F_{\rho_{1,1}(\Gamma)}(z)$ using \eqref{eq:F_G(z)}. 

\item Compare the values from the previous item, and obtain all the coincidences, which induce isospectralities by \eqref{eq:encode}.
\end{enumerate}
\end{algorithm}

All the computational examples produced by Algorithm~\ref{algorithm} for $N \le 30000$ are listed in Table~\ref{table}.
Notably, all pairs in Table~\ref{table} arise from the construction in Theorem~\ref{thm:contraejemplos}.

\begin{remark}
The step \eqref{item:F_G(z)} in Algorithm~\ref{algorithm} can be optimized replacing the ring of the complex power series $\C[[z]]$ by a finite field $\F$ of huge prime order (e.g.\ $10^{20}$). 
One has to replace $z$ a random element in $F$ and 
$\xi_m$ (resp.\ $\xi_n$) in the polynomials $P_{a,b}(z),\bar P_{a,b}(z)$ (see Remark~\ref{rem:det(z-g)}) by some 
$m$-th (resp.\ $n$-th) root of unity in $\F$. 
\end{remark}

\section{Open questions}

We conclude this article with some open questions.

\begin{question}\label{question-d=4,6}
What is the smallest possible dimension for a pair of isospectral spherical space forms with non-isomorphic fundamental groups?
\end{question}

The next table summarizes what is known about this question. 

\medskip

\begin{tabular}{ccl}
\begin{tabular}{c}
dim \\ $2d-1$
\end{tabular} & $d$ & 
\begin{tabular}{l}
Existence of isospectral spherical space forms \\
with non-cyclic fundamental groups.
\end{tabular}
\\ \hline
\rule{0pt}{13pt}%
3 & 2 & No. Isospectrality implies isometry  \cite[Thm.~I]{Ikeda80_3-dimI}.
\\
5 & 3 & No. Non-isometric isospectrality implies lens spaces  \cite[Thm.~3.9]{Ikeda80_3-dimII}.
\\
7 & 4 & \quad \textbf{Unknown}. 
\\
9 & 5 & No. Isospectrality implies isomorphic fundamental groups \cite[Thm.~3.1]{Ikeda80_3-dimII}.
\\
11 & 6 & \quad \textbf{Unknown}. 
\\
13 & 7 & No. Isospectrality implies isomorphic fundamental groups \cite[Thm.~3.1]{Ikeda80_3-dimII}.
\\
15 & 8 & Yes, see Remark~\ref{rem:smallest-dim}. 
\\
\end{tabular}

\medskip

Note that in dimensions 7 and 11 there are spherical space forms with fundamental groups of other types (see \cite[\S7.5]{Wolf-book}). 
One may concentrate in Type I groups, though it is not known either whether there are isospectral spherical space forms of dimension 7 or 11 with non-isomorphic fundamental groups of Type I. 
In dimension 7 (resp.\ 11), the possible Type I groups have $d=2,4$ (resp.\ $d=2,3,6$). 
The following result, combined with Theorem~\ref{thm:main1}, excludes the case $d=2$. 

\begin{proposition}\label{prop:d=2=>isomorphic}
Two Type I groups of the form $\Gamma_2(m,n,r_i)$ for $i=1,2$ are isomorphic. 
\end{proposition}

\begin{proof}
We have that $r_i^2\equiv 1\pmod m$, so $m\mid (r_i-1)(r_i+1)$. 
However, $m$ is relatively prime to $r_i-1$ by Lemma~\ref{lem:Burnside1}, thus $m\mid r_i+1$. 
Hence $r_1\equiv -1\equiv r_2\pmod m$, and the assertion follows by Proposition~\ref{prop:TypeI-uptoisomorphism}. 
\end{proof}

\begin{question}\label{question-d-odd}
	Are there isospectral spherical space forms of dimension $\equiv 1\pmod 4$ with non-isomorphic fundamental groups? 
\end{question}

The fundamental group of a spherical space form of dimension $\equiv 1\pmod 4$ is necessarily of Type I. 
Therefore, Question~\ref{question-d-odd} is equivalent via Theorem~\ref{thm:main1} to ask whether there are non-isomorphic Type I groups of the form $\Gamma_i:= \Gamma_{d}(m,n,r_i)$ for $i=1,2$, with $d$ odd, such that $\Spec\big(S^{2q-1}/\rho^{(1)}(\Gamma_1)\big) = \Spec\big(S^{2q-1}/\rho^{(2)}(\Gamma_2)\big)$ for some $2q$-dimensional fixed point free representations $\rho^{(1)},\rho^{(2)}$ of $\Gamma_1,\Gamma_2$ respectively. 

We mentioned in Subsection~\ref{subsec:ejemploscomputacionales} that all computational examples in Table~\ref{table} can be constructed by Theorem~\ref{thm:contraejemplos}.

\begin{question}
Does Theorem~\ref{thm:contraejemplos} explain all pairs of isospectral spherical space forms with non-isomorphic fundamental groups?
\end{question}

The following question was asked in \cite[Question 5.13]{LMR-survey}, and is related to Remark~\ref{rem:stronglyisospectral}. 

\begin{question}
Are there isospectral spherical space forms with non-cyclic fundamental groups that are not strongly isospectral? 
\end{question}

\bibliographystyle{plain}

\end{document}